\documentclass[12pt]{amsart}

\usepackage{latexsym,amsmath,amssymb,amsthm,amsfonts}

\usepackage[numeric]{amsrefs}

\usepackage[hidelinks]{hyperref}

\usepackage[capitalise]{cleveref}

\newcommand{\R}{{\mathbb{R}}}
\newcommand{\Z}{{\mathbb{Z}}}

\newcommand{\eps}{\varepsilon}
\newcommand{\N}{{\mathbb{N}}}
\newcommand{\E}{{\mathbb{E}}}

\newcommand{\diam}{{\rm diam}\,}

\renewcommand{\phi}{{{\varphi}}}

\newtheorem{theorem}{Theorem}
\newtheorem{lemma}[theorem]{Lemma}
\newtheorem{proposition}[theorem]{Proposition}
\newtheorem*{proposition*}{Proposition}

\newtheorem{corollary}[theorem]{Corollary}

\theoremstyle{remark}

\newtheorem{openproblem}{Open problem}

\title{On Gr{\"u}nbaum's problem for symmetric configurations}

\author{A. Arman}
 \address{Department of Mathematics, University of Manitoba, Winnipeg, MB, R3T 2N2, Canada}
 \email{andrew0arman@gmail.com}
 
\author{A.\ Bondarenko}
 \address{Department of Mathematical Sciences, Norwegian University of Science and
 	Technology, NO-7491 Trondheim, Norway}
 \email{andriybond@gmail.com}
 \thanks{The second author was supported in part by Grant 334466 of the Research Council of Norway.}

\author{A.\ Prymak}
 \address{Department of Mathematics, University of Manitoba, Winnipeg, MB, R3T 2N2, Canada}
 \email{prymak@gmail.com}
 \thanks{The third author was supported by NSERC of Canada Discovery Grant RGPIN-2026-06488.}

\author{D.\ Radchenko}
 \address{Institut des Hautes \'Etudes Scientifiques, CNRS, Laboratoire Alexander Grothendieck, 35 route de Chartres, Bures-sur-Yvette 91440, France}
 \email{danradchenko@gmail.com}
\thanks{The fourth author was supported by ERC Starting Grant No. 101078782.}

\date{\today}
\subjclass[2020]{Primary 52C17; Secondary 52A20, 52C35, 94A24.}
\keywords{Gr{\"u}nbaum's covering problem, equal-diameter covering, rate-distortion, hypersimplices, symmetric configurations.}

\begin{document}

\begin{abstract}
Let $g_n$ be the largest number of Euclidean balls of diameter $1$ which may be needed to cover a set of diameter $1$ in $\R^n$.  We study this problem for finite sets invariant under all coordinate permutations.  We prove that the exponential growth rate in this symmetric problem can be characterized exactly as a finite-alphabet squared-error rate-distortion supremum $\alpha_0$.

Specialized to the two-point case, i.e., for subsets of Boolean cubes, this gives the explicit lower bound
	\[g_n\ge (1.160235457\ldots-o(1))^n,\]
improving the previous best bound $(2/\sqrt3-o(1))^n$.  Using Fix's Gaussian characterization of the rate-distortion problem, we give a numerical three-point construction with exponent base greater than $1.160497831$.  Finally, we show that $\alpha_0$ is not attained by any finitely supported distribution.
\end{abstract}

\maketitle

\section{Introduction}
Let $\R^n$ be the $n$-dimensional Euclidean space.  For a bounded set $A\subset\R^n$, let $N_{\rm cov}(A,r)$ be the minimum number of Euclidean balls of radius $r$ whose union contains $A$.  Gr\"unbaum's equal-diameter covering problem asks for the asymptotic behaviour of
\[
 g_n:=\sup_{\diam A=1} N_{\rm cov}(A,1/2),
\]
where the supremum is over bounded sets $A\subset\R^n$; see, for example, Gr\"unbaum~\cite{G} and Bourgain--Lindenstrauss~\cite{BL}.  Equivalently, $g_n$ is the largest number of balls of diameter $1$ which may be needed to cover a set of diameter $1$.

This problem is related to Borsuk's problem, with the important difference that the covering pieces are prescribed to be balls of the same diameter as the original set.  Bourgain and Lindenstrauss~\cite{BL} proved that every set of diameter $1$ in $\R^n$ can be covered by $(\sqrt{3/2}+o(1))^n$ balls of diameter $1$.  They also constructed sets requiring exponentially many such balls~\cite{BL}*{Eqn.~(8)}.  The lower bound was later improved in~\cite{ABP}, where finite sets of diameter $1$ were constructed which cannot be covered by $(2/\sqrt3-o(1))^n$ balls of diameter $1$.  The latter construction starts from a discrete random spherical set and then applies a removal method.

Let $S_n$ denote the symmetric group, acting on $\R^n$ by permutations of coordinates.  We define the symmetric Gr\"unbaum constant by
\[
g_n^{\rm sym}:=\sup\left\{N_{\rm cov}(A,\diam(A)/2):\ A\subset\R^n\text{ finite},\ \diam(A)>0,\ \sigma A=A\text{ for every }\sigma\in S_n\right\}.
\]
The definition is scale invariant and is equivalent, after rescaling, to the diameter-one version with $N_{\rm cov}(A,1/2)$.  In particular, $g_n^{\rm sym}\le g_n$.

Put
\[
H(p)=-p\log p-(1-p)\log(1-p), \qquad 0\le p\le 1,
\]
with the convention $0\log0=0$, and define
\begin{equation}\label{eq:alpha-bin}
 \alpha_{\rm bin}:=
 \max_{0<a<1/2}
 \exp\left[H(a)-H\left(\frac{1-\sqrt{1-2a}}{2}\right)\right].
\end{equation}
Our first result is the following explicit bound.

\begin{theorem}\label{thm:main}
For the symmetric Gr\"unbaum constants, we have
	\[\liminf_{n\to\infty}\bigl(g_n^{\rm sym}\bigr)^{1/n}\ge \alpha_{\rm bin}.\]
Consequently,
	\[\liminf_{n\to\infty}g_n^{1/n}\ge \alpha_{\rm bin}.\]
Moreover,
	\[ \alpha_{\rm bin}=1.160235457\ldots .\]
\end{theorem}

The construction in Theorem~\ref{thm:main} is given by the constant-weight layer of the Boolean cube
	\[M_{n,a}=\left\{x\in\{0,1\}^n:\ \sum_{j=1}^n x_j=\lfloor an\rfloor\right\},
	 \qquad 0<a<1/2.\]
Equivalently, for $k=\lfloor an\rfloor$, this is the vertex set of the hypersimplex
	\[\Delta(k,n)=\left\{x\in[0,1]^n:\ \sum_{j=1}^n x_j=k\right\}.\]
The proof in Section~\ref{sec:hypersimplex} shows that for $a=a^\ast=0.212856445\ldots,$ for which the maximum in~\eqref{eq:alpha-bin} is attained, the sets $X_n=M_{n,a}$ satisfy $N_{\mathrm{cov}}(X_n,\diam(X_n)/2) = \alpha_{\mathrm{bin}}^{n+o(n)}$.

The hypersimplex construction uses only two coordinate values.  More generally, one may fix a finite real alphabet $\{x_1,\ldots,x_m\}$ and take all permutations of a vector containing prescribed proportions of the values $x_i$.  This leads to a finite alphabet rate-distortion formulation.  For a finitely supported real random variable $X$, let
	\[D(X)=\sup_{X_1,X_2}\mathbb E|X_1-X_2|^2,\]
where the supremum is over all couplings $(X_1,X_2)$ with $X_1\sim X_2\sim X$.  Let
	\[R_X(\Delta):=\inf_{Y:\ \mathbb E|X-Y|^2\le \Delta} I(X;Y)\]
be the squared-error rate-distortion function, where the infimum is over finitely supported real reconstruction variables $Y$ coupled with $X$.  Define
\begin{equation}\label{eqn:alpha0}
\alpha(X):=R_X(D(X)/4),
\qquad
\alpha_0:=\sup_X\alpha(X),
\end{equation}
where the supremum is over all finitely supported real random variables $X$.

\begin{theorem}\label{thm:rd-symmetric}
With $\alpha_0$ as in~\eqref{eqn:alpha0},
\[
\lim_{n\to\infty}\frac1n\log g_n^{\rm sym}=\alpha_0.
\]
Moreover, for every finitely supported non-constant real random variable $X$ there exist finite $S_n$-invariant sets $T_n\subset\R^n$ such that
\[
N_{\rm cov}(T_n,\diam(T_n)/2)\ge \exp\left(n(\alpha(X)+o(1))\right).
\]
\end{theorem}

\begin{corollary}\label{cor:rd-lower}
For the unrestricted Gr\"unbaum constants,
\[
\liminf_{n\to\infty}\frac1n\log g_n\ge \alpha_0.
\]
\end{corollary}

For two-point distributions, Theorem~\ref{thm:rd-symmetric} recovers the hypersimplex exponent in Theorem~\ref{thm:main}.  Using Fix's characterization of the squared-error rate-distortion problem~\cite{Fix}, in Section~\ref{sec:fix} we give numerically a three-point distribution with
\[
 \exp(\alpha(X))>1.160497831351.
\]
Consequently,
\[
\lim_{n\to\infty}\bigl(g_n^{\rm sym}\bigr)^{1/n}
=e^{\alpha_0}>1.160497831351,
\]
and the same lower bound holds for $\liminf_{n\to\infty}g_n^{1/n}$.  This slightly improves the hypersimplex value. Our numerical computations with larger finite alphabets suggest that any further improvement will be extremely small, but the exact distributional optimum is unclear.

\begin{openproblem}\label{prob:alpha0}
Determine $\alpha_0$ and describe the behaviour of extremizing sequences.
\end{openproblem}

Although the value of $\alpha_0$ is unknown, the supremum in~\eqref{eqn:alpha0} is not attained on a finite alphabet.

\begin{proposition}\label{prop:no-finite-maximizer}
For every finitely supported non-constant real random variable $X$, there exists a finitely supported real random variable $X'$ with $\alpha(X')>\alpha(X)$.  Consequently, $\alpha_0$ is not attained by any finitely supported distribution.
\end{proposition}

The paper is organized as follows.  In Section~\ref{sec:hypersimplex} we prove Theorem~\ref{thm:main}. In Section~\ref{sec:info} we describe the information-theoretic reformulation and prove Theorem~\ref{thm:rd-symmetric} and Corollary~\ref{cor:rd-lower}.  In Section~\ref{sec:fix} we recall Fix's Gaussian reformulation of the rate-distortion problem, describe the numerical three-point example, and prove Proposition~\ref{prop:no-finite-maximizer}.

\section{The hypersimplex construction}\label{sec:hypersimplex}
We use $\log$ to denote the natural logarithm, and we keep the convention $0\log0=0$. In this section we prove Theorem~\ref{thm:main}.  For a fixed $a\in(0,1/2)$, put $k=\lfloor an\rfloor$ and
\[
M_{n,a}:=\left\{x\in\{0,1\}^n:\ \sum_{j=1}^n x_j=k\right\}.
\]
This is the constant-weight layer of the Boolean cube with weight $k$.  For all sufficiently large~$n$ we have $k\le n/2$, and therefore two points of $M_{n,a}$ can have disjoint supports.  Hence
	\[\diam(M_{n,a})=\sqrt{2k}=\bigl(\sqrt{2a}+o(1)\bigr)\sqrt n.\]

The proof of Theorem~\ref{thm:main} is split into two parts.  First we estimate the largest possible intersection of $M_{n,a}$ with a single ball of radius $\diam(M_{n,a})/2$. This is an entropy maximization problem. Then we optimize the resulting exponent over $a$.

We shall use the following standard consequence of Stirling's formula.
\begin{lemma}\label{lem:stirling}
For fixed $A>0$ and $0\le B\le A$,
	\[\log\binom{\lfloor An\rfloor}{\lfloor Bn\rfloor}
	=An\bigl(H(B/A)+o(1)\bigr),\qquad n\to\infty.\]
\end{lemma}

\subsection{The covering ratio}
For $b>0$, define
\[
F_n(a,b):=\sup_{y\in\R^n}\bigl|M_{n,a}\cap B(y,b\sqrt n)\bigr|.
\]
We shall use
\[
F(n,a):=F_n\left(a,\sqrt{\frac{k}{2n}}\right)
=\sup_{y\in\R^n}\bigl|M_{n,a}\cap B(y,\diam(M_{n,a})/2)\bigr|.
\]
Let
\[
N(n,a):=N_{\rm cov}\bigl(M_{n,a},\diam(M_{n,a})/2\bigr).
\]
Since each ball of radius $\diam(M_{n,a})/2$ contains at most $F(n,a)$ points of $M_{n,a}$,
\begin{equation}\label{eq:Rna_def}
N(n,a)\ \ge\ \frac{|M_{n,a}|}{F(n,a)}=:R_n(a).
\end{equation}

Let $a^\ast=0.212856445\ldots$ be the unique solution in $(0,1/2)$ of
\begin{equation}\label{eq:astar_equation}
2\sqrt{1-2a}\,\log\frac{1-a}{a}=\log\frac{1+\sqrt{1-2a}}{1-\sqrt{1-2a}}.
\end{equation}
Theorem~\ref{thm:main} follows by applying the estimates below to $M_{n,a^\ast}$.

\subsection{The entropy maximization for one ball}
For $m\in\N$, denote by $D_{n,m}$ the set of vectors in $\R^n$ having at most $m$ distinct coordinates, and set
\[
F_n^{(m)}(a,b):=\sup_{y\in D_{n,m}}\bigl|M_{n,a}\cap B(y,b\sqrt n)\bigr|.
\]
We shall use the following elementary quantization to reduce arbitrary centers to finite-level centers: for every $a\in(0,1)$, $b>0$, $n\in\N$, and $m\ge1$,
\begin{equation}\label{eq:finite_level_reduction}
F_n(a,b)\le F_n^{(m+1)}\left(a,b+\frac1m\right).
\end{equation}
Indeed, first project the center $y$ onto the cube $[0,1]^n$, obtaining $y'$; this can only decrease its distance to every point of $\{0,1\}^n$.  Then round each coordinate of $y'$ to the nearest point of the grid $\{0,1/m,\ldots,1\}$.  The resulting vector $z$ has at most $m+1$ distinct coordinates and satisfies $\|y'-z\|_2\le \sqrt n/m$.  Hence
\[
B(y,b\sqrt n)\cap\{0,1\}^n\subset B\left(z,\left(b+\frac1m\right)\sqrt n\right)\cap\{0,1\}^n,
\]
which proves \eqref{eq:finite_level_reduction}.

For $m\in\N$, define
	\[\Lambda_m(a,b):=
	\sup\left\{\sum_{i=1}^m p_iH(r_i):\ \sum_{i=1}^m p_i=1,\ \sum_{i=1}^m p_ir_i=a,\ \sum_{i=1}^m p_ir_i(1-r_i)\le b^2\right\},\]
where $p_i\ge0$ and $r_i\in[0,1]$.  Also put
	\[\Lambda(a,b):=
	\sup\left\{\int_0^1 H(r)\,d\mu(r):\ \int_0^1 r\,d\mu(r)=a,\ \int_0^1 r(1-r)\,d\mu(r)\le b^2\right\},\]
where the supremum is over all probability measures $\mu$ on $[0,1]$.

We will use the variational formula
\begin{equation}\label{eq:variational_formula}
\lim_{n\to\infty}\frac1n\log F_n(a,b)=\Lambda(a,b),
\end{equation}
valid for every fixed $a\in(0,1)$ and $b>0$.  To prove it, first fix $m$.  If $y\in D_{n,m}$, partition $\{1,\ldots,n\}$ into blocks on which $y$ is constant, say equal to $c_i$ on a block of size $n_i$.  A point of $M_{n,a}$ has some numbers $\ell_i$ of ones in these blocks, with $\sum_i\ell_i=k$.  For this type the number of points is $\prod_i\binom{n_i}{\ell_i}$, and its squared distance from $y$ is
\[
\sum_i\left(\ell_i(1-c_i)^2+(n_i-\ell_i)c_i^2\right).
\]
For fixed densities $p_i=n_i/n$ and $r_i=\ell_i/n_i$, the best choice of $c_i$ is $c_i=r_i$, and the corresponding normalized squared distance is $\sum_i p_ir_i(1-r_i)$.  Since the number of possible types is polynomial in $n$ for fixed $m$, Lemma~\ref{lem:stirling} gives
\[
\lim_{n\to\infty}\frac1n\log F_n^{(m)}(a,b)=\Lambda_m(a,b).
\]
The floor in $k=\lfloor an\rfloor$ changes only the $o(n)$ term.

The upper bound for arbitrary centers follows from \eqref{eq:finite_level_reduction}: for every $m$,
\[
\limsup_{n\to\infty}\frac1n\log F_n(a,b)
\le \Lambda_{m+1}\left(a,b+\frac1m\right).
\]
Letting $m\to\infty$ gives at most $\Lambda(a,b)$.  Indeed, any almost extremizing sequence of finitely supported measures for the right-hand side has a weakly convergent subsequence, and the limit satisfies the constraints defining $\Lambda(a,b)$.

For the lower bound, choose a finitely supported probability measure $\mu=\sum_i p_i\delta_{r_i}$ with mean $a$, with $\int r(1-r)\,d\mu(r)<b^2$, and with $\int H(r)\,d\mu(r)>\Lambda(a,b)-\eps$.  Such a strictly feasible measure is obtained by mixing an almost optimizer with the Bernoulli measure $(1-a)\delta_0+a\delta_1$ by an arbitrarily small amount.  After a rational approximation of the $p_i$ and $r_i$, choose blocks of sizes $n_i=p_in+o(n)$ and choose $\ell_i=r_in_i+o(n)$ ones in block $i$, with $\sum_i\ell_i=k$.  Taking the center to be $r_i$ on block $i$, all points of this type lie in $B(y,b\sqrt n)$ for all large $n$, and their number has exponent $\sum_i p_iH(r_i)+o(1)$.  Letting $\eps\downarrow0$ proves \eqref{eq:variational_formula}.

\begin{lemma}\label{prop:explicit_f}
For every $a\in(0,1/2)$,
\[
\lim_{n\to\infty}\frac1n\log F(n,a)
=H\!\left(u(a)\right),
\qquad
u(a):=\frac{1-\sqrt{1-2a}}{2}.
\]
\end{lemma}

\begin{proof}
Since $\sqrt{k/(2n)}\to\sqrt{a/2}$, monotonicity in the radius, \eqref{eq:variational_formula}, and continuity of the compact variational problem in~$b$ reduce the claim to computing $\Lambda(a,\sqrt{a/2})$.  Let $\mu$ be feasible in the definition of this quantity.  Put $s(r)=r(1-r)$ and
\[
G(s):=H\!\left(\frac{1-\sqrt{1-4s}}{2}\right),\qquad 0\le s\le \frac14.
\]
Since $H(r)=H(1-r)$, we have $H(r)=G(s(r))$ for all $r\in[0,1]$.  Writing $t=\sqrt{1-4s}$, direct differentiation gives
\[
G'(s)=\frac1t\log\frac{1+t}{1-t}>0
\]
and
\[
G''(s)=\frac{(1-t^2)\log\frac{1+t}{1-t}-2t}{2s\,t^3}\le0.
\]
The last inequality follows from
\[
\log\frac{1+t}{1-t}=2\int_0^t\frac{dx}{1-x^2}\le \frac{2t}{1-t^2},\qquad 0<t<1.
\]
Thus $G$ is increasing and concave.  Jensen's inequality yields
\[
\int H(r)\,d\mu(r)=\int G(s(r))\,d\mu(r)
\le G\left(\int s(r)\,d\mu(r)\right)
\le G(a/2).
\]
Let $u=u(a)$ be the smaller root of $u(1-u)=a/2$; then $G(a/2)=H(u)$, so this gives the upper bound.

For the reverse inequality, choose $p\in[0,1]$ such that $pu+(1-p)(1-u)=a$, and set
\[
\mu=p\delta_u+(1-p)\delta_{1-u}.
\]
Then $\int r\,d\mu(r)=a$, while $r(1-r)=a/2$ on the support of $\mu$.  Moreover $H(u)=H(1-u)$, so $\int H\,d\mu=H(u)$.  Hence $\Lambda(a,\sqrt{a/2})=H(u)$.
\end{proof}

\begin{proof}[Proof of Theorem~\ref{thm:main}]
By Lemma~\ref{lem:stirling} and Lemma~\ref{prop:explicit_f},
\[
\frac1n\log R_n(a)
=H(a)-H\!\left(u(a)\right)+o(1),
\qquad n\to\infty.
\]
Set $\Phi(a):=H(a)-H(u(a))$ and $t=\sqrt{1-2a}$.  Since $u(a)=(1-t)/2$,
\[
\Phi'(a)=\log\frac{1-a}{a}-\frac{1}{2t}\log\frac{1+t}{1-t}.
\]
Thus the critical point equation is precisely \eqref{eq:astar_equation}.  Moreover,
\[
\Phi'(a)\to+\infty\quad(a\downarrow0),\qquad \Phi'(a)\to-1\quad(a\uparrow1/2),
\]
and
\[
\Phi''(a)=\frac{1-3t^2}{t^2(1-t^4)}-\frac{1}{2t^3}\log\frac{1+t}{1-t}<0,
\]
for $0<t<1$: if $t^2\ge1/3$ this is immediate, while if $t^2<1/3$, then
\[
\log\frac{1+t}{1-t}>2t>\frac{2t(1-3t^2)}{1-t^4}.
\]
Hence $\Phi$ is strictly concave and has a unique maximizer $a^\ast\in(0,1/2)$.  Numerically,
\[
a^\ast=0.212856445\ldots,\qquad
\Phi(a^\ast)=0.148622964\ldots,
\]
and therefore
\[
\exp(\Phi(a^\ast))=1.160235457\ldots .
\]
The lower bound for $N(n,a^\ast)$ follows from \eqref{eq:Rna_def}.  Since $M_{n,a^\ast}$ is $S_n$-invariant, the definition of $g_n^{\rm sym}$ gives
\[
 g_n^{\rm sym}\ge \exp\left(n\bigl(\Phi(a^\ast)+o(1)\bigr)\right).
\]
This proves Theorem~\ref{thm:main}, with $\alpha_{\rm bin}=\exp(\Phi(a^\ast))$.
\end{proof}

\section{An information-theoretic reformulation}\label{sec:info}

\subsection{Setup}
We now pass from the hypersimplex to an $S_n$-orbit built from an arbitrary finite set of real values. The information-theoretic notation used below is a convenient way to express elementary counting facts.  We describe these facts first.

Let
	\[\mathcal X=\{x_1,\ldots,x_m\},\]
and fix positive numbers $p_1,\ldots,p_m$ with $\sum_i p_i=1$.  Let $X$ be the random variable given by
	\[\mathbb P(X=x_i)=p_i.\]
For every $n$ we choose integers $n_i=n_i(n)$ such that
	\[\sum_i n_i=n,
	\qquad
	p_i^{(n)}:=\frac{n_i}{n}\longrightarrow p_i.\]
(For definiteness we could choose $n_i$ satisfying $n_1+\dots+n_j=\lfloor n(p_1+\dots+p_j)\rfloor$.)
We write $X^{(n)}$ for the random variable with probabilities $p_i^{(n)}$. Consider the set
	\begin{equation}\label{eq:permutation-orbit}
	T_n(X):=\left\{z\in\mathcal X^n:\ \#\{j:\ z_j=x_i\}=n_i\text{ for every }i\right\}\subset\R^n.
	\end{equation}
Thus $T_n(X)$ is an $S_n$-orbit of a vector with $n_i$ copies of $x_i$, $i=1,\dots,m$.  The random variables $X$ and $X^{(n)}$ are simply compact notation for the limiting and finite coordinate proportions.

For a finitely supported random variable $Z$, with probabilities $q_1,\ldots,q_s$, we define by 
	\[H(Z)=-\sum_{i=1}^s q_i\log q_i\]
its entropy. For us its relevance is purely combinatorial: by Stirling's formula, the number of sequences of length $n$ in which the $i$-th symbol occurs $nq_i$ times is
	\[\binom{n}{nq_1,\dots,nq_s} = \exp\bigl(nH(Z)+o(n)\bigr).\]
In particular,
	\[\log |T_n(X)|=nH(X^{(n)})+o(n),
	\qquad
	|T_n(X)|=\exp\bigl(nH(X)+o(n)\bigr).\]

Now fix a point $z\in T_n(X)$ and a possible center $y\in\R^n$, and suppose that the center $y$ also uses a fixed finite (and independent of $n$) set of coordinates. Let $J$ be a uniformly distributed random variable with values in $\{1,\ldots,n\}$ and set
	\[Z=z_J, \qquad Y=y_J.\]
Then the joint distribution of $(Z,Y)$ records the frequencies of the coordinate pairs $(z_j,y_j)$, and
	\[\frac1n\|z-y\|_2^2=\mathbb E|Z-Y|^2.\]
Moreover, for a given $y$, we can calculate the number of vectors $z\in T_n(X)$ with a given joint distribution $(Z,Y)$. If we write $a_1,\ldots,a_r$ for the distinct coordinate values of $y$, and put
	\[N_{ij}:=n\mathbb P(Z=x_i,Y=a_j)\,,\]
then for each $j$, the number of choices of coordinates $z_k$ for which $y_k=a_j$ is given by the multinomial coefficient $\binom{N_{1j}+\dots+N_{mj}}{N_{1j},\ldots,N_{mj}}$. Therefore the total number of compatible $z$ is
	\[\prod_{j=1}^r
	\binom{N_{1j}+\dots+N_{mj}}{N_{1j},\ldots,N_{mj}}=\exp\bigl(nH(Z\mid Y)+o(n)\bigr),\]
where
	\[H(Z\mid Y)=\sum_y\mathbb P(Y=y)H(Z\mid Y=y)\]
is the conditional entropy. 

Combining the above counts, suggests that the number of $y$'s needed to cover $T_n(X)$ by sets of $z$'s with fixed joint distribution $(Z,Y)$ is
	\[\exp\big(nI(Z;Y)+o(n)\big)\,,\]
where
	\[I(Z;Y):=H(Z)-H(Z\mid Y)\]
is the so-called mutual information.

A ball of squared radius $n\Delta$ permits only pair frequencies satisfying $\mathbb E|X-Y|^2\le\Delta$. To make the intersection with a single ball as large as possible, one maximizes $H(X\mid Y)$ under this constraint, or equivalently minimizes $I(X;Y)$.  This leads to
	\[R_X(\Delta):=\inf_{Y:\ \mathbb E|X-Y|^2\le\Delta}I(X;Y),\]
where the infimum is over all joint distributions of $(X,Y)$ with the prescribed distribution of $X$ and with $Y$ finitely supported and real-valued.  As we will see below, $R_X(\Delta)$ is the asymptotic exponential rate of the number of balls of radius $\sqrt{n\Delta}$ needed to cover $T_n(X)$.

The diameter of $T_n(X)$ also has a probabilistic description.  For two points of $T_n(X)$, a uniformly chosen coordinate gives a pair $(X_1,X_2)$ whose two marginal distributions are both $X^{(n)}$; conversely, every $\frac1n\Z_{\ge0}$-valued table of pair frequencies with these marginals is realized by two points of the orbit. Hence, if
	\[D(X):=\sup_{X_1,X_2}\mathbb E|X_1-X_2|^2,\]
where the supremum is over all joint couplings $(X_1,X_2)$ with $X_i\sim X$, then
	\[\frac1n\diam(T_n(X))^2\longrightarrow D(X).\]
A ball with half the diameter therefore has normalized squared radius $nD(X)/4$.  Accordingly,
	\[\alpha(X)=R_X(D(X)/4),
	\qquad
	\alpha_0=\sup_X\alpha(X).\]

\subsection{Detailed proof sketch of Theorem~\ref{thm:rd-symmetric}}

We explain the two directions separately.

\noindent\emph{The lower bound.}
Fix a finitely supported non-constant $X$, form $T_n=T_n(X)$ as in~\eqref{eq:permutation-orbit}, and put
\[
\delta_n:=\frac1n\diam(T_n)^2.
\]
The calculations in the setup give
\[
|T_n|=\exp\bigl(nH(X)+o(n)\bigr),
\qquad
\delta_n\longrightarrow D(X).
\]
It remains to control an arbitrary ball center, which need not use a fixed number of coordinate values.  Let $S=T_n\cap B(y,\diam(T_n)/2)$, choose $U=(U_1,\ldots,U_n)$ uniformly from $S$, and independently choose $J$ uniformly from $\{1,\ldots,n\}$.  Every point of $S$ has the same coordinate multiplicities, so $U_J$ has distribution $X^{(n)}$.  Moreover,
\[
\mathbb E|U_J-y_J|^2
=\frac1{|S|n}\sum_{u\in S}\|u-y\|_2^2
\le \frac{\delta_n}{4}.
\]
Since $y_J$ is determined by $J$, knowing $J$ can only leave less uncertainty about $U_J$ than knowing $y_J$ alone.  Therefore
\begin{align*}
	\log|S|=H(U)
	&\le\sum_{j=1}^n H(U_j)
	=nH(U_J\mid J)\\
	&\le nH(U_J\mid y_J)\\
	&\le n\bigl(H(X^{(n)})-R_{X^{(n)}}(\delta_n/4)\bigr).
\end{align*}
Here the last inequality is precisely the definition of $R_{X^{(n)}}$.  Comparing this bound with the size of $T_n$ gives
\[
N_{\rm cov}(T_n,\diam(T_n)/2)
\ge \exp\left(nR_{X^{(n)}}(\delta_n/4)+o(n)\right)
=\exp\left(n(\alpha(X)+o(1))\right).
\]
The last equality is just continuity of a minimization over a compact finite-dimensional set: the probabilities $p_i^{(n)}$ tend to $p_i$, the distortion $\delta_n/4$ tends to $D(X)/4$, and one may take the second variable to have at most $m+1$ possible values, all in a fixed compact interval.  Taking the supremum over $X$ gives the required lower bound by $\alpha_0$.

\smallskip
\noindent\emph{The upper bound.}
Recall that $S_n$ is the symmetric group, acting on $\R^n$ by coordinate permutations. Let $A\subset\R^n$ be finite, $S_n$-invariant, and normalized so that $\diam A=1$.  We first reduce $A$ to a small collection of permutation orbits. Translating in the direction $(1,\dots,1)$, we may assume that the barycenter of $A$ is at the origin. Since $\diam A=1$, we may henceforth assume that $A\subset B(0,1)$.

Fix a small $\eps>0$.  For every $x\in A$, round coordinates of absolute value at most $\eps/(2\sqrt n)$ to zero, and round each remaining absolute value down to the nearest number of the form
\[
\frac{\eps}{2\sqrt n}(1+\eps/2)^k,
\]
while keeping its sign.  This changes the vector by at most $\eps$ in Euclidean norm and leaves only $O_\eps(\log n)$ possible coordinate values. Hence the number of rounded orbits is at most
\[
(n+1)^{O_\eps(\log n)}=\exp(o(n)).
\]
Choose a set $\mathcal Q$ containing one representative from each rounded orbit. Then $|\mathcal Q|=\exp(o(n))$, and $A$ lies within distance $\eps$ of
\[
\bigcup_{q\in\mathcal Q}S_nq.
\]

It remains to cover one orbit.  Let $X_q$ record its coordinate frequencies after scaling by~$\sqrt n$:
\[
\mathbb P(X_q=\sqrt n\,r)
=\frac1n\#\{i:q_i=r\}.
\]
Since $S_nq$ lies within distance $\eps$ of $A$, we have $\diam(S_nq)\le1+2\eps$. The definition of $X_q$ gives
\[
D(X_q)=\diam(S_nq)^2\le(1+2\eps)^2.
\]

For any $\rho<1/2$, choose a joint distribution of $(X_q,Y)$ that nearly realizes $R_{X_q}(\rho^2)$ and round its probabilities to integer coordinate counts. Let $d_q$ be the number of points of $S_nq$ having the prescribed joint frequencies with a fixed center.  The conditional-entropy count from the setup gives
	\[d_q=\exp\bigl(nH(X_q\mid Y)+o(n)\bigr).\]
By a standard random covering argument (choose centers uniformly and use biregularity to see that each point is covered with probability $d_q/|S_nq|$),
\[
N_{\rm cov}(S_nq,\rho+o(1))
\le \bigl(1+\log|S_nq|\bigr)\frac{|S_nq|}{d_q}
=\exp\bigl(nR_{X_q}(\rho^2)+o(n)\bigr).
\]
This estimate is uniform over $q\in\mathcal Q$: the $O_\eps(\log n)$ possible coordinate values contribute only $o(n)$ to the multinomial estimates, while the bound $\mathbb E X_q^2=\|q\|_2^2\le1$ allows the few very large values to be kept unchanged.

Now take $\rho=1/2-2\eps$ and put $\Delta_0=D(X_q)/4$.  By definition,
\[
R_{X_q}(\Delta_0)=\alpha(X_q)\le\alpha_0.
\]
If $\Delta_0\le\rho^2$, monotonicity gives $R_{X_q}(\rho^2)\le\alpha_0$.  Otherwise, take a nearly optimal pair $(X_q,Y)$ at distortion $\Delta_0$, let $B$ be an independent Bernoulli random variable with $\mathbb P(B=1)=\lambda$, and define $Y'=X_q$ when $B=1$ and $Y'=Y$ when $B=0$.  Setting $\lambda=1-\rho^2/\Delta_0$, we have
\[
\mathbb E|X_q-Y'|^2\le(1-\lambda)\Delta_0=\rho^2.
\]
By the data-processing inequality~\cite{CT}*{Section~2.8},
\[
I(X_q;Y')\le I(X_q;B,Y')
=(1-\lambda)I(X_q;Y)+\lambda H(X_q).
\]
The bound $D(X_q)\le(1+2\eps)^2$ shows that $\lambda=O(\eps)$. Using the fact that the support of $X_q$ forms a geometric progression, together with $\mathbb E X_q^2\le1$, one can show that $H(X_q)=O(\log(1/\eps))$ holds uniformly in $n$. Therefore,
\[
R_{X_q}(\rho^2)\le\alpha_0+O(\eps\log(1/\eps)).
\]
After enlarging the covering balls by the initial approximation error $\eps$, their radius is at most~$1/2$.  Since $|\mathcal Q|=\exp(o(n))$, the whole set $A$ can therefore be covered by
\[
\exp\bigl(n(\alpha_0+O(\eps\log(1/\eps))+o(1))\bigr)
\]
balls of radius $1/2$, where $o(1)\to0$ as $n\to\infty$ for fixed $\eps$.  First let $n\to\infty$ and then $\eps\downarrow0$.  This proves the upper bound and completes the proof. \qed

\begin{proof}[Proof of Corollary~\ref{cor:rd-lower}]
	This follows at once from Theorem~\ref{thm:rd-symmetric}, since $g_n\ge g_n^{\rm sym}$.
\end{proof}

When $X$ is supported on $\{0,1\}$ with $\mathbb P(X=1)=a\le1/2$, the sets $T_n(X)$ are the layers $M_{n,a}$ with $k/n\to a$.  In this case $D(X)=2a$, and Lemma~\ref{prop:explicit_f} gives
	\[\alpha(X)=H(a)-H\left(\frac{1-\sqrt{1-2a}}{2}\right),\]
so the two-point case of the rate-distortion formulation recovers Theorem~\ref{thm:main}.

\section{A three-point distribution and a Gaussian reformulation}\label{sec:fix}
We recall Fix's characterization~\cite{Fix} of the optimizer in the squared-error rate-distortion problem from Section~\ref{sec:info}.  Fix's statement concerns a prescribed distortion level; throughout this section we denote this level by $\Delta$.

Let $\mathcal X$ be the finite support of $X$, and let $P(x)=\mathbb P(X=x)$.  In the non-trivial range $0<\Delta<\operatorname{Var}(X)$, there exists a number $s>0$ and a strictly positive function $\alpha\colon\mathcal X\to (0,\infty)$ such that
\[
C(y)=\sum_{x\in\mathcal X}\alpha(x)P(x)e^{-s(x-y)^2}
\]
satisfies $C(y)\le1$ for all $y\in\R$.  An optimal reconstruction distribution $Q$ is supported on the finite contact set $\mathcal Y=C^{-1}(1)$ and satisfies
\[
\frac1{\alpha(x)}=\sum_{y\in\mathcal Y}Q(y)e^{-s(x-y)^2}.
\]
The corresponding optimal channel is
\[
\mathbb P(Y=y\mid X=x)=\alpha(x)Q(y)e^{-s(x-y)^2},
\]
and its rate is
\[
\sum_{x\in\mathcal X}P(x)\log\alpha(x)-s\Delta.
\]
The distortion may also be recovered from
\[
\sum_yQ(y)C''(y)
=\sum_{x,y}\alpha(x)P(x)Q(y)e^{-s(x-y)^2}
  \bigl(4s^2(x-y)^2-2s\bigr)
=2s\bigl(2s\mathbb E|X-Y|^2-1\bigr).
\]

We shall also use the dual half of this characterization.  If $s>0$ and $a\colon\mathcal X\to(0,\infty)$ satisfy
\[
\sum_{x\in\mathcal X}P(x)a(x)e^{-s(x-y)^2}\le1
\qquad\text{for every }y\in\R,
\]
then
\begin{equation}\label{eq:rd-dual-lower}
R_X(\Delta)\ge \sum_{x\in\mathcal X}P(x)\log a(x)-s\Delta.
\end{equation}
Indeed, this is the log-sum inequality applied to the joint law of $(X,Y)$ and the sub-probability measure
\[
P(x)P_Y(y)a(x)e^{-s(x-y)^2}.
\]
Equality holds for the optimal data above.

The following three-point numerical candidate leads to a small improvement over the exponent from \cref{thm:main}:

\begin{align*}
    X &= \begin{bmatrix}
        0 \\
        2.1361583129963432932551182078530101126630888598883 \\
        4.8391028430568877914799153960214290951290015744938
    \end{bmatrix} \\[1ex]
    P &= \begin{bmatrix}
        0.7854464210172769953753097061250384632886118945593 \\
        0.2134421263533836981569289942024470544115451329875 \\
        0.0011114526293393064677612996725144822998429724532
    \end{bmatrix} \\[1ex]
    \alpha &= \begin{bmatrix}
        1.1604978313513987171135503645129178447210657924979 \\
        4.2305522982531473104212020077852105800097089594071 \\
        885.93383293483472942565598323570751225946650948497
    \end{bmatrix}
\end{align*}
\begin{align*}
    Y &= \begin{bmatrix}
        0.2651059566293051272482657065985716878727329953196 \\
        1.9114017952529115119565287436748171426551645601336 \\
        4.7946984041506117301899432909379702105894338153745
    \end{bmatrix} \\[1ex]
    Q &= \begin{bmatrix}
        0.8811670582579179994364471379686748123678760530378 \\
        0.1186302321008139002145964133077049698542681435763 \\
        0.0002027096412681003489564487236202177778558033857
    \end{bmatrix}\\[1ex]
    s &= 0.5669229113540285972530639842885832122672672583115
\end{align*}

The displayed $Y$ and $Q$ numerically satisfy the optimality equations to more than $45$ decimal places, to the displayed accuracy $D(X)=2$ and $\E|X-Y|^2=1/2=D(X)/4$. We have
	\[\exp(\alpha(X))>1.160497831351.\]
It is not hard to obtain a certified bound, but we decided against doing this here, since, as we show below, a three-point distribution cannot be optimal.

\subsection*{Proof of Proposition~\ref{prop:no-finite-maximizer}}
\begin{proof}
It is enough to consider a non-constant finitely supported random variable $X$.  Put $D=D(X)$ and $\Delta=D/4$.  If $\alpha(X)=0$, choose any two-point random variable $X'$ for which $\alpha(X')>0$, as described at the end of Section~\ref{sec:info}.  We may therefore assume that $\alpha(X)>0$, and hence that we are in the non-trivial case $0<\Delta<\mathrm{Var}(X)$. Let $P$ be the probability law of $X$, and let $s>0$, $\alpha$, and
\[
C(y)=\sum_x P(x)\alpha(x)e^{-s(x-y)^2}
\]
be the optimal Gaussian data for $R_X(\Delta)$.  Thus $C(y)\le1$, the contact set $\mathcal Y=C^{-1}(1)$ is finite and non-empty, and
\[
\alpha(X)=R_X(D/4)=\sum_x P(x)\log\alpha(x)-\frac{sD}{4}.
\]
Set
\[
A:=\sum_x P(x)\log\alpha(x),
\qquad
b:=\max\mathcal Y,
\]
and choose $B$ so large that $\operatorname{supp}X\subset[-B,B]$ and $\mathcal Y\subset[-B,B]$.

Let $z>B+2$ and define
	\[a_z:=e^{s(z-b-1)^2},
	\qquad
	\varepsilon_z:=\frac{1}{2a_z}.\]
Let $X_z$ have the probability law
	\[P_z=(1-\varepsilon_z)P+\varepsilon_z\delta_z.\]
We claim that, for all sufficiently large $z$, the same value of $s$, together with the weights $\alpha(x)$ on the old support and the weight $a_z$ at the new atom $z$, is feasible for the dual inequality~\eqref{eq:rd-dual-lower} for $X_z$.  In other words,
\begin{equation}\label{eq:new-atom-feasible}
(1-\varepsilon_z)C(y)+\varepsilon_z a_z e^{-s(z-y)^2}\le1
\qquad\text{for all }y\in\R.
\end{equation}
Indeed, choose a neighborhood of $\mathcal Y$, of radius less than $1/2$ and so small that on it $C\ge1/2$.  On this neighbourhood the new term in~\eqref{eq:new-atom-feasible} is $o(\varepsilon_z)$ uniformly as $z\to\infty$, whereas
\[
1-(1-\varepsilon_z)C(y)=1-C(y)+\varepsilon_z C(y)\ge \frac{\varepsilon_z}{2}.
\]
On a fixed compact set away from the chosen neighbourhood, $C\le1-\eta$ for some $\eta>0$, while the new term tends to zero uniformly.  Finally, outside a sufficiently large compact set we have $C\le1/3$, and the new term is at most $\varepsilon_z a_z=1/2$.  This proves~\eqref{eq:new-atom-feasible}.

Applying~\eqref{eq:rd-dual-lower} to $X_z$ with distortion $D(X_z)/4$ gives
\begin{equation}\label{eq:alpha-z-lower}
\alpha(X_z)
\ge
(1-\varepsilon_z)A+\varepsilon_z\log a_z-\frac{s}{4}D(X_z).
\end{equation}
We also need a simple estimate on the new diameter parameter.  In any coupling of $X_z$ with itself, the total mass of pairs involving the new atom $z$ is at most $2\varepsilon_z$.  The remaining coupling can be completed, using only old atoms, to a coupling of $X$ with itself; hence its contribution is at most $D(X)$.  Therefore
\begin{equation}\label{eq:Dz-bound}
D(X_z)\le D(X)+2\varepsilon_z(z+B)^2.
\end{equation}
Combining~\eqref{eq:alpha-z-lower} and~\eqref{eq:Dz-bound}, we obtain
\[
\alpha(X_z)-\alpha(X)
\ge
\varepsilon_z\left(\log a_z-A-\frac{s}{2}(z+B)^2\right).
\]
Since $\log a_z=s(z-b-1)^2$, the expression in parentheses is positive for all sufficiently large $z$.  Hence $\alpha(X_z)>\alpha(X)$ for such $z$.

This proves that every non-constant finitely supported $X$ can be improved.  If a constant random variable were a maximizer, then $\alpha_0=0$.  This is impossible, since the two-point case described at the end of Section~\ref{sec:info} gives a finitely supported variable with $\alpha(X)=\log\alpha_{\rm bin}>0$.  Hence no finitely supported random variable can attain $\alpha_0$.
\end{proof}

\section*{AI use disclosure}

Generative AI was used as a writing and editing tool in the preparation of this manuscript.

\end{document}